\documentclass[11pt,amssymb]{amsart}
\usepackage{amssymb}
\usepackage[mathscr]{eucal}

\newtheorem{thm}{Theorem}

\newtheorem{prop}{Proposition}

\begin{document}

\title{Residue maps, Azumaya algebras, and buildings}

\author[I.A.~Rapinchuk]{Igor A. Rapinchuk}

\address{Department of Mathematics, Michigan State University, East Lansing, MI 48824, USA}

\email{rapinchu@msu.edu}

\begin{abstract}
The goal of this note is to give an explicit formula for the residues of  twists of the matrix algebra in terms of the twisting cocycle.
Combined with the Fixed Point Theorem for actions of finite groups on affine buildings, this leads to a quick proof of the well-known characterization of unramified algebras in terms of Azumaya algebras.
\end{abstract}

\maketitle

\section{Introduction}\label{S:Int}

Let $\mathcal{K}$ be a field complete with respect to a discrete valuation $v$. Suppose $\mathcal{L}/\mathcal{K}$ is a finite unramified Galois extension of degree $n \geq 1$ with Galois group $\mathrm{Gal}(\mathcal{L}/\mathcal{K})$ and let $w$ be the extension of $v$ to $\mathcal{L}$. One then defines a {\it residue map} in Galois cohomology
\footnote{We use the standard notions associated with Galois cohomology --- cf. \cite{Serre-GC};
in particular, for a finite Galois extension $L/K$ with Galois group $\Gamma = \mathrm{Gal}(L/K)$ and a commutative $\Gamma$-module $M$, the cohomology groups $H^i(\Gamma , M)$ will be denoted $H^i(L/K , M)$, and for an algebraic $K$-group $G$, the 1-cohomology set $H^1(\Gamma , G(L))$ will be denoted $H^1(L/K , G)$.}:
$$
\rho \colon H^2(\mathcal{L}/\mathcal{K} , \mathcal{L}^{\times}) \longrightarrow H^1(\mathcal{L}/\mathcal{K} , \mathbb{Q}/\mathbb{Z})
$$
(see, e.g., \cite[\S3]{Wadsworth}).
We note that the target of $\rho$ is typically denoted $H^1(\ell/k , \mathbb{Q}/\mathbb{Z})$, where $\ell/k$ is the corresponding extension of residue fields; however, in our situation, $\mathrm{Gal}(\ell/k) \simeq \mathrm{Gal}(\mathcal{L}/\mathcal{K})$ since $\mathcal{L}/\mathcal{K}$ is unramified. Furthermore, composing $\rho$ with the classical isomorphism $\varepsilon \colon \mathrm{Br}(\mathcal{L}/\mathcal{K}) \to H^2(\mathcal{L}/\mathcal{K} , \mathcal{L}^{\times})$ for the relative Brauer group based on factor sets (cf. \cite[Theorem 4.13]{FD}, \cite[\S4]{IR-Br}), we obtain the usual {\it residue map} $r \colon \mathrm{Br}(\mathcal{L}/\mathcal{K}) \to H^1(\mathcal{L}/\mathcal{K} , \mathbb{Q}/\mathbb{Z})$ --- see \S \ref{S:Res}. The elements of $\mathrm{Br}(\mathcal{L}/\mathcal{K})$, however, can be described not only in terms of factor sets, but also as the Brauer classes $[A(c)]$ of central simple $\mathcal{K}$-algebras obtained from the matrix algebra $M_n(\mathcal{K})$ by twisting using 1-cocycles $c \in Z^1(\mathcal{L}/\mathcal{K} , \mathrm{PGL}_n)$, cf. \cite[\S4.4]{GS}. The goal of this note is to give an explicit description of the residue $r([A(c)])$ in terms of the cocycle $c$ that does not seem to have ever been recorded in the literature. To formulate the result, we let $\theta \colon \mathrm{GL}_n(\mathcal{L}) \to \mathrm{PGL}_n(\mathcal{L})$ denote the canonical homomorphism. Given a cocycle $c = (c_{\sigma}) \in Z^1(\mathcal{L}/\mathcal{K} , \mathrm{PGL}_n(\mathcal{L}))$, we pick $\tilde{c}_{\sigma} \in \theta^{-1}(c_{\sigma})$.
\begin{thm}\label{T:1}
$r([A(c)])$ is represented by $b = (b_{\sigma}) \in Z^1(\mathcal{L}/\mathcal{K} , \mathbb{Q}/\mathbb{Z})$ with
$$
b_{\sigma} = -\frac{w(\det \tilde{c}_{\sigma})}{n} (\mathrm{mod}\: \mathbb{Z}).
$$
\end{thm}

Next, we recall that a class $[A] \in \mathrm{Br}(\mathcal{L}/\mathcal{K})$ (or the corresponding algebra $A$) is said to be {\it unramified} if $r([A]) = 0$.
Combining  Theorem \ref{T:1} with the Fixed Point Theorem for the actions of finite groups on buildings we obtain a new quick proof of the following well-known characterization of unramified algebras.
\begin{thm}\label{T:2}
A central simple $\mathcal{K}$-algebra $A$ such that $[A] \in \mathrm{Br}(\mathcal{L}/\mathcal{K})$ is unramified if and only if there exists an Azumaya algebra $\mathcal{A}$ over the valuation ring $\mathcal{O}$ of $\mathcal{K}$ such that $\mathcal{A} \otimes_{\mathcal{O}} \mathcal{K} \simeq A$.
\end{thm}

\section{Coboundary map and factor sets}\label{S:Cob}

Let $L/K$ be a finite Galois extension of degree $n \geq 2$. We  consider the following exact sequence of $\mathrm{Gal}(L/K)$-groups
$$
1 \to L^{\times} \longrightarrow \mathrm{GL}_n(L) \stackrel{\theta}{\longrightarrow} \mathrm{PGL}_n(L) \to 1,
$$
and let
$$
\Delta \colon H^1(L/K , \mathrm{PGL}_n(L)) \longrightarrow H^2(L/K , L^{\times})
$$
denote the corresponding coboundary map. Thus, if $c \in H^1(L/K , \mathrm{PGL}_n(L))$ is represented by a cocycle $(c_{\sigma}) \in Z^1(L/K , \mathrm{PGL}_n(L))$, then picking arbitrary lifts $\tilde{c}_{\sigma} \in \theta^{-1}(c_{\sigma})$, the element $d = \Delta(c)$ is represented by the 2-cocycle $(d_{\sigma , \tau})$ with
\begin{equation}\label{E:200}
d_{\sigma , \tau} I_n = \tilde{c}_{\sigma} \cdot \sigma(\tilde{c}_{\tau}) \cdot \tilde{c}_{\sigma\tau}^{-1}.
\end{equation}
We note that a standard argument using twisting and Hilbert's Theorem 90 shows that $\Delta$ is injective (cf. \cite[\S2.2.3]{PlRa}).

On the other hand, one can consider the central simple $K$-algebra $A(c)$ obtained by twisting the matrix algebra $M_n(K)$ using the cocycle $c$. We recall that $A(c)$ can be described as the fixed subalgebra of the twisted action of $\Gamma = \mathrm{Gal}(L/K)$ on $M_n(L)$, i.e.
\begin{equation}\label{E:twist}
A(c) = \left\{ x \in M_n(L) \ \vert \ c_{\sigma}(\sigma(x)) = x \ \ \text{for all} \ \ \sigma \in \Gamma \right\},
\end{equation}
where for $x = (x_{ij}) \in M_n(L)$, we set $\sigma(x) = (\sigma(x_{ij}))$ and consider the natural action of $\mathrm{PGL}_n(L)$ on $M_n(L)$ by (inner) automorphisms (thus, $c_{\sigma}$ is identified with $\mathrm{Int}\: \tilde{c}_{\sigma}$ in the previous notations). Then $A(c) \otimes_K L \simeq M_n(L)$, so one can consider the Brauer class $[A(c)] \in \mathrm{Br}(L/K)$ and the corresponding 2-cohomology class $a = \varepsilon([A(c)]) \in H^2(L/K , L^{\times})$ given by the associated factor set. We will need a relation between the cohomology classes $a$ and $d$. This relation is described in Exercise 2 in Ch. X, \S 5 of \cite{Serre-LF}; however, to the best of our knowledge, the details have never been published, so we include a complete argument.
\begin{prop}\label{P:1}
We have $d = a^{-1}$.
\end{prop}
\begin{proof}
We will interpret $L^n$ as an $L$-vector space equipped with the (canonical) basis $e_{\sigma}$ indexed by $\sigma \in \Gamma$, and consider the action of $\Gamma$ on $L^n$ that fixes these basic elements, as well as the induced action on  $\mathrm{GL}_n(L) = \mathrm{Aut}_L(L^n)$. Define $\tilde{p}_{\sigma} \in \mathrm{GL}_n(L)$ by
\begin{equation}\label{E:14}
\tilde{p}_{\sigma}(e_{\tau}) = d_{\sigma , \tau} e_{\sigma\tau}.
\end{equation}
We claim that
\begin{equation}\label{E:13}
\tilde{p}_{\sigma} \cdot \sigma(\tilde{p}_{\tau}) \cdot \tilde{p}_{\sigma\tau}^{-1} = d_{\sigma , \tau} I_n \ \ \text{for all} \ \ \sigma , \tau \in \Gamma.
\end{equation}
Indeed, for any $\gamma \in \Gamma$, we have
$$
(\tilde{p}_{\sigma} \sigma(\tilde{p}_{\tau}))(e_{\gamma}) = \tilde{p}_{\sigma}(\sigma(d_{\tau , \gamma}) e_{\tau\gamma}) = \sigma(d_{\tau , \gamma}) d_{\sigma , \tau\gamma} e_{\sigma(\tau\gamma)}.
$$
On the other hand,
$$
d_{\sigma , \tau} \tilde{p}_{\sigma\tau}(e_{\gamma}) = d_{\sigma , \tau} d_{\sigma \tau, \gamma} e_{(\sigma\tau)\gamma}.
$$
Then (\ref{E:13}) follows from the 2-cocycle condition
\begin{equation}\label{E:cocycle}
\sigma(d_{\tau , \gamma}) d_{\sigma , \tau\gamma} = d_{\sigma , \tau} d_{\sigma\tau , \gamma} \ \ \text{for all} \ \ \sigma, \tau, \ \text{and} \ \gamma \in \Gamma.
\end{equation}
Equation (\ref{E:13}) implies that  $(p_{\sigma} := \theta(\tilde{p}_{\sigma}))_{\sigma \in \Gamma}$ is a cocycle in $Z^1(L/K , \mathrm{PGL}_n)$ such that for the corresponding cohomology class $p \in H^1(L/K , \mathrm{PGL}_n)$, we have $\Delta(p) = d$. Then the injectivity of $\Delta$ yields $p = c$, hence $A(p) \simeq A(c)$ and $a = \varepsilon([A(c)]) = \varepsilon([A(p)])$.

Using (\ref{E:twist}) with $c$ replaced by $p$, we obtain the following description:
$$
A(p) = \{ x \in M_n(L) \ \vert \ \tilde{p}_{\sigma} \sigma(x) \tilde{p}_{\sigma}^{-1} = x \ \ \text{for all} \ \ \sigma \in \Gamma \}.
$$
Rewriting the defining conditions in the form $x\tilde{p}_{\sigma} = \tilde{p}_{\sigma} \sigma(x)$
and labelling the entries of $x$ as $x_{\rho , \tau}$ so that $x(e_{\tau}) = \sum_{\rho} x_{\rho , \tau} e_{\rho}$, one easily checks that
$$
A(p) = \{ x = (x_{\rho , \tau}) \in M_n(L) \ \vert \ d_{\sigma , \tau} x_{\sigma\rho , \sigma\tau} = d_{\sigma , \rho} \sigma(x_{\rho, \tau}) \ \text{for all} \ \sigma \in \Gamma \}.
$$
Using this description, one deduces from the cocycle condition (\ref{E:cocycle}) that the elements $q_{\rho} \in \mathrm{GL}_n(L)$ for $\rho \in \Gamma$ defined by
$$
q_{\rho}(e_{\sigma}) = d_{\sigma , \rho} e_{\sigma\rho}
$$
(compare with (\ref{E:14})) lie in $A(p)$, hence so do their inverses $r_{\rho} := q_{\rho}^{-1}$. For $\ell \in L$, we let
$$
\varphi(\ell) = \mathrm{diag}(\sigma(\ell))_{\sigma \in \Gamma}.
$$
It is straightforward to check that $\varphi$ defines a $K$-embedding $L \hookrightarrow A(p)$ and that
\begin{equation}\label{E:17}
r_{\sigma} \varphi(\ell) r_{\sigma}^{-1} = \varphi(\sigma(\ell)) \ \ \text{for all} \ \ \sigma \in \Gamma \ \ \text{and} \ \ \ell \in L.
\end{equation}
Set $\Lambda = \varphi(L)$. Then a standard argument using (\ref{E:17}) shows that the elements $r_{\sigma}$ for $\sigma \in \Gamma$ are linearly independent over $\Lambda$, hence
$$
A(p) = \bigoplus_{\sigma \in \Gamma} \Lambda r_{\sigma}
$$
(see, for example, \cite[Lemma 6]{IR-Br}).
It follows from the definition of the factor set $(b_{\sigma , \tau} \in L^{\times})_{\sigma, \tau \in \Gamma}$ for $A(p)$ associated with the extension $\Lambda/K$ together with our choice of the elements $r_{\sigma}$ that
$$
r_{\sigma} r_{\tau} = \varphi(b_{\sigma , \tau}) r_{\sigma\tau},
$$
or equivalently
$$
q_{\tau}q_{\sigma} = q_{\sigma\tau} \varphi(b_{\sigma , \tau}) ^{-1}.
$$
Applying both sides to $e_{\iota}$, where $\iota$ is the identity element of $\Gamma$, we obtain
$$
(q_{\tau} q_{\sigma})(e_{\iota}) = q_{\tau} (d_{\iota , \sigma} e_{\sigma}) = d_{\iota ,\sigma} d_{\sigma , \tau} e_{\sigma\tau}
$$
and
$$
(q_{\sigma\tau}\varphi(b_{\sigma , \tau})^{-1})(e_{\iota}) = q_{\sigma\tau}(b_{\sigma , \tau}^{-1} e_{\iota}) = b_{\sigma , \tau}^{-1} d_{\iota , \sigma\tau} e_{\sigma\tau}.
$$
Thus,
$$
b_{\sigma , \tau} = \frac{d_{\iota , \sigma\tau}}{d_{\iota , \sigma} d_{\sigma , \tau}} = \frac{d_{\iota , \sigma\tau}}{d_{\sigma , \tau} d_{\iota , \sigma\tau}} = d_{\sigma , \tau}^{-1}
$$
in view of the cocycle condition (\ref{E:cocycle}) for $(d_{\sigma , \tau})$. Then the corresponding cohomology class $b \in H^2(L/K , L^{\times})$, which coincides with $a$ since $\varepsilon$ is well-defined, equals $d^{-1}$, as required.
\end{proof}

\section{Residue maps: proof of Theorem \ref{T:1}}\label{S:Res}

Let $\mathcal{K}$ be a field that is complete with respect to a discrete valuation $v$. Suppose $\mathcal{L}$ is an {\it unramified} Galois extension of $\mathcal{K}$ of (finite) degree $n$ with Galois group $\Gamma = \mathrm{Gal}(\mathcal{L}/\mathcal{K})$, and let $w$ be the extension of $v$ to $\mathcal{L}$. We now recall the construction of the {\it residue map}
$$
\rho \colon H^2(\mathcal{L}/\mathcal{K} , \mathcal{L}^{\times}) \longrightarrow H^1(\mathcal{L}/\mathcal{K} , \mathbb{Q}/\mathbb{Z}),
$$
where $\mathbb{Q}/\mathbb{Z}$ is equipped with the trivial $\Gamma$-action. The valuation  $w$ can be regarded as a $\Gamma$-homomorphism $\mathcal{L}^{\times} \to \mathbb{Z}$ (where $\Gamma$ acts trivially on $\mathbb{Z}$), hence induces a group homomorphism
$$
\tilde{\rho} \colon H^2(\mathcal{L}/\mathcal{K} , \mathcal{L}^{\times}) \longrightarrow H^2(\mathcal{L}/\mathcal{K} , \mathbb{Z}).
$$
Then the residue map $\rho$ is given by $\rho = \nu^{-1} \circ \tilde{\rho}$ where
$$
\nu \colon H^1(\mathcal{L}/\mathcal{K} , \mathbb{Q}/\mathbb{Z}) \longrightarrow H^2(\mathcal{L}/\mathcal{K} , \mathbb{Z})
$$
is the isomorphism induced by the coboundary map arising from the following exact sequence of trivial $\Gamma$-modules:
$$
0 \to \mathbb{Z} \longrightarrow \mathbb{Q} \stackrel{\pi}{\longrightarrow} \mathbb{Q}/\mathbb{Z} \to 0.
$$
So, if $a \in H^2(\mathcal{L}/\mathcal{K} , \mathcal{L}^{\times})$ is represented by a 2-cocycle $(a_{\sigma , \tau})_{\sigma, \tau \in \Gamma} \in Z^2(\mathcal{L}/\mathcal{K} , \mathcal{L}^{\times})$, then for $b = \rho(a) \in H^1(\mathcal{L}/\mathcal{K} , \mathbb{Q}/\mathbb{Z})$, the class $\nu(b)$ is represented by the cocycle $(w(a_{\sigma , \tau}))_{\sigma, \tau \in \Gamma}$. This means that $b$ is represented by a cocycle $(b_{\sigma})_{\sigma \in \Gamma} \in Z^1(\mathcal{L}/\mathcal{K} ,  \mathbb{Q}/\mathbb{Z})$ such that for some lifts $\tilde{b}_{\sigma} \in \pi^{-1}(b_{\sigma})$, we have the identity
\begin{equation}\label{E:100}
\tilde{b}_{\sigma} + \sigma(\tilde{b}_{\tau}) - \tilde{b}_{\sigma\tau} = \tilde{b}_{\sigma} + \tilde{b}_{\tau} - \tilde{b}_{\sigma\tau} = w(a_{\sigma , \tau}) \ \ \text{for all} \ \ \sigma, \tau \in \Gamma.
\end{equation}
We then define the residue map $r \colon \mathrm{Br}(\mathcal{L}/\mathcal{K}) \to H^1(\mathcal{L}/\mathcal{K} , \mathbb{Q}/\mathbb{Z})$ as the composition of $\rho$ with the isomorphism $\varepsilon \colon \mathrm{Br}(\mathcal{L}/\mathcal{K}) \to H^2(\mathcal{L}/\mathcal{K} , \mathcal{L}^{\times})$ given by factor sets.

Now suppose that a cohomology class $c \in H^1(\mathcal{L}/\mathcal{K} , \mathrm{PGL}_n)$ is represented by  $$(c_{\sigma})_{\sigma \in \Gamma} \in Z^1(\mathcal{L}/\mathcal{K} , \mathrm{PGL}_n(\mathcal{L})).$$ As in \S \ref{S:Cob}, we pick arbitrary lifts $\tilde{c}_{\sigma} \in \theta^{-1}(c_{\sigma})$, so that $d = \Delta(c)$ is represented by $(d_{\sigma , \tau}) \in Z^2(\mathcal{L}/\mathcal{K} , \mathcal{L}^{\times})$ satisfying (\ref{E:200}). Set $$f_{\sigma} = \frac{w(\det \tilde{c}_{\sigma})}{n} \in \mathbb{Q}.$$ Taking determinants in (\ref{E:200}) and applying $w$, we obtain
\begin{equation}\label{E:300}
w(d_{\sigma , \tau}) = f_{\sigma} +  f_{\tau} - f_{\sigma\tau}.
\end{equation}
According to Proposition \ref{P:1}, the class $\varepsilon([A(c)]) \in H^2(\mathcal{L}/\mathcal{K} , \mathcal{L}^{\times})$ is represented by the cocycle $a_{\sigma , \tau} := d_{\sigma , \tau}^{-1}$. Then it follows from (\ref{E:300}) that (\ref{E:100}) is satisfied with
$$
\tilde{b}_{\sigma} = - f_{\sigma}.
$$
Then $r([A(c)])$ is represented by the cocycle $(-f_{\sigma}(\mathrm{mod}\: \mathbb{Z}))$, which proves Theorem \ref{T:1}.

\section{Azumaya algebras and buildings: proof of Theorem \ref{T:2}}

We continue with the notations and conventions of the previous section. In particular, $\mathcal{L}/\mathcal{K}$ is an unramified extension of degree $n$ with Galois group $\Gamma = \mathrm{Gal}(\mathcal{L}/\mathcal{K}).$ Furthermore, we let $\mathcal{O} = \mathcal{O}_{\mathcal{K}}$ (resp., $\mathcal{O}_{\mathcal{L}}$) denote the valuation ring of $\mathcal{K}$ (resp., $\mathcal{L}$). Let $\mathscr{B}$ be the Bruhat-Tits building associated with the group $\mathrm{PGL}_n(\mathcal{L})$. We recall that $\mathscr{B}$ is a simplicial complex whose vertices correspond to homothety  classes of
$\mathcal{O}_{\mathcal{L}}$-lattices in the vector space $\mathcal{L}^n$, referring the reader to \cite[\S3.4]{PlRa} and \cite[Chapter 19]{Garrett} for an explicit description of other attributes of $\mathscr{B}$. We now introduce the following subgroup of $\mathrm{GL}_n(\mathcal{L})$:
$$
\widetilde{H} = \{ h \in \mathrm{GL}_n(\mathcal{L}) \ \vert \ w(\det h) \equiv 0 (\mathrm{mod}\: n) \},
$$
and let $H$ denote the image of $\widetilde{H}$ is $\mathrm{PGL}_n(\mathcal{L})$. Then $H$ is precisely the subgroup of $\mathrm{PGL}_n(\mathcal{L})$ of {\it type-preserving} projective transformations of $\mathscr{B}$ (see \cite[A.1.3]{AbBr} and \cite[2.5]{ARR} for a discussion of type-preserving automorphisms of buildings). We now point out that the natural action of $\Gamma = \mathrm{Gal}(\mathcal{L}/\mathcal{K})$ on $\mathscr{B}$ is also type-preserving. Indeed,   let $e_1, \ldots , e_n$ be the standard basis of $\mathcal{L}^n$, and let $p_0$ be the vertex of $\mathscr{B}$ corresponding to the lattice $\mathcal{O}_{\mathcal{L}}e_1 + \cdots + \mathcal{O}_{\mathcal{L}}e_n$. Let $p$ be any other vertex of $\mathscr{B}$. Since the group $\mathrm{PGL}_n(\mathcal{L})$ obviously acts transitively on the vertices of $\mathscr{B}$, we can find $g \in \mathrm{PGL}_n(\mathcal{L})$ such that $p = g(p_0)$. Then any $\sigma \in \Gamma$ fixes $p_0$, and consequently $\sigma(p) = h(p)$ where $h = \sigma(g) g^{-1}$. But if $\tilde{g} \in \mathrm{GL}_n(\mathcal{L})$ is a lift of $g$, then $w(\det(\sigma(\tilde{g}) \tilde{g}^{-1})) = 0$. Thus, $h \in H$ and our claim follows from the fact that $H$ acts by type-preserving transformations. Thus, we have a natural action of the group $\mathrm{PGL}_n(\mathcal{L}) \rtimes \Gamma$ on $\mathscr{B}$, with the subgroup $H \rtimes \Gamma$ acting by type-preserving transformations.

Next, we recall the following formalism. Given a finite group $\Delta$ and a $\Delta$-group $G$, a map $c \colon \Delta \to G$ is a 1-cocycle if and only if the map
$$
f_c \colon \Delta \longrightarrow G \rtimes \Delta, \ \ \ f_c(\delta) = (c(\delta) , \delta),
$$
is a group homomorphism, and furthermore two cocycles $c_1 , c_2 \in Z^1(\Delta , G)$ are cohomologous if and only if the corresponding group homomorphisms $f_{c_1}$ and $f_{c_2}$ are conjugate by an element of $G$ (cf. \cite[Lemma 3.1]{ARR}). We will apply this to $\Delta = \Gamma$ and $G = \mathrm{PGL}_n(\mathcal{L})$. So, for a given $c \in Z^1(\mathcal{L}/\mathcal{K} , \mathrm{PGL}_n)$, we let $f_c \colon \Gamma \to \mathrm{PGL}_n(\mathcal{L}) \rtimes \Gamma$ denote the corresponding homomorphism as well as the $c$-{\it twisted action} of $\Gamma$ on $\mathscr{B}$ obtained by composing $f_{c}$ with the standard action  of $\mathrm{PGL}_n(\mathcal{L}) \rtimes \Gamma$ on $\mathscr{B}$ introduced above.

Starting the proof of Theorem \ref{T:2}, we first note that since the field $\mathcal{K}$ discretely valued, it follows from \cite[Lemma 9.5]{Saltman} that if $B$ is a central simple $\mathcal{K}$-algebra such that $[B] \in \mathrm{Br}(\mathcal{K})$ is in the image of the natural map $\mathrm{Br}(\mathcal{O}) \to \mathrm{Br}(\mathcal{K})$, then any maximal $\mathcal{O}$-order of $B$ is an Azumaya algebra over $\mathcal{O}.$ From this, it easily follows that for two Brauer-equivalent central simple $\mathcal{K}$-algebras $A_1$ and $A_2$, an Azumaya $\mathcal{O}$-algebra $\mathcal{A}_i$ such that $\mathcal{A}_i \otimes_{\mathcal{O}} \mathcal{K} \simeq A_i$ exists for $i = 1$ if and only if it exists for $i = 2$. Consequently, it suffices
to prove Theorem \ref{T:2} for an algebra of the form $A(c)$ for some $c \in Z^1(\mathcal{L}/\mathcal{K} , \mathrm{PGL}_n)$.

Assume that the algebra $A(c)$ is unramified. Applying Theorem \ref{T:1}, we conclude that $c_{\sigma} \in H$ for all $\sigma \in \Gamma$ in the above notations. Thus, the image of the group homomorphism $f_c$ is contained in $H \rtimes \Gamma$, and consequently, the $c$-twisted action of $\Gamma$ on $\mathscr{B}$ is type-preserving. The latter fact, in conjunction with the Fixed Point Theorem for actions of finite groups on affine Bruhat-Tits buildings (cf. \cite[Theorem 11.23]{AbBr}) implies the existence of a fixed {\it vertex} in $\mathscr{B}$ for this action (cf. \cite[Lemma 2.9]{ARR}). As $\mathrm{PGL}_n(\mathcal{L})$ acts on vertices transitively, replacing $f_c$ by a $\mathrm{PGL}_n(\mathcal{L})$-conjugate homomorphism, which amounts to passing from $c$ to an equivalent cocycle, we may assume that $f_c(\Gamma)$ fixes $p_0$. Explicitly, this means that
$$
c_{\sigma}(\sigma(p_0)) = p_0 \ \ \text{for all} \ \ \sigma \in \Gamma.
$$
But the (usual) action of $\Gamma$ fixes $p_0$, so this condition yields that $c_{\sigma}$ lies in the stabilizer of $p_0$ in $\mathrm{PGL}_n(\mathcal{L})$, which is $\mathrm{PGL}_n(\mathcal{O}_{\mathcal{L}})$ for all $\sigma \in \Gamma$.  The natural action of $\mathrm{PGL}_n(\mathcal{O}_{\mathcal{L}})$ on $M_n(\mathcal{O}_{\mathcal{L}})$ enables us to consider the $c$-twisted action of $\Gamma$, and we let $\mathcal{A}(c)$ denote the $\mathcal{O}$-subalgebra of fixed elements. Since $\mathcal{L}/\mathcal{K}$ is unramified, the discriminant of the ring extension $\mathcal{O}_{\mathcal{L}}/\mathcal{O}$ is a unit, which implies that $\mathcal{A}(c) \otimes_{\mathcal{O}} \mathcal{O}_{\mathcal{L}} \simeq M_n(\mathcal{O}_{\mathcal{L}})$ (cf. \cite[Lemma C.1]{CRR}; more generally, this follows from the existence of descent for the Galois ring extension $\mathcal{O}_{\mathcal{L}}/\mathcal{O}$, cf. \cite[Proposition 5.1.12]{K}). We conclude that $\mathcal{A}(c)$ is an Azumaya $\mathcal{O}$-algebra such that $\mathcal{A}(c) \otimes_{\mathcal{O}} \mathcal{K} \simeq A(c)$, as required.

Conversely, suppose $A$ is a central simple $\mathcal{K}$-algebra such that $[A] \in \mathrm{Br}(\mathcal{L}/\mathcal{K})$ and there exists an Azumaya $\mathcal{O}$-algebra $\mathcal{A}$ such that $\mathcal{A} \otimes_{\mathcal{O}} \mathcal{K} \simeq A$. Without loss of generality, we may assume that $A$ has degree $n$. Then
$$
(\mathcal{A} \otimes_{\mathcal{O}} \mathcal{O}_{\mathcal{L}}) \otimes_{\mathcal{O}_{\mathcal{L}}} \mathcal{L} \simeq (\mathcal{A} \otimes_{\mathcal{O}} \mathcal{K}) \otimes_{\mathcal{K}} \mathcal{L} \simeq M_n(\mathcal{L}),
$$
implying that $\mathcal{A} \otimes_{\mathcal{O}} \mathcal{O}_{\mathcal{L}} \simeq M_n(\mathcal{O}_{\mathcal{L}})$ (since the natural map $\mathrm{Br}(\mathcal{O}_{\mathcal{L}}) \to \mathrm{Br}(\mathcal{L})$ is injective --- cf. \cite[Ch. IV, Corollary 2.6]{Milne-EC}). It follows that $\mathcal{A}$ can be obtained from $M_n(\mathcal{O})$, and hence $A$ can be obtained from $M_n(\mathcal{K})$, by twisting using a 1-cocycle on $\Gamma$ with values in $\mathrm{PGL}_n(\mathcal{O}_{\mathcal{L}})$ (which is the automorphism group of the Azumaya $\mathcal{O}_{\mathcal{L}}$-algebra $M_n(\mathcal{O}_{\mathcal{L}})$). Then it immediately follows from Theorem \ref{T:1} that $A$ is unramified.

\vskip2mm

\noindent {\small {\bf Acknowledgements.} I would like to thank Gopal Prasad for helpful comments and the anonymous referee for a careful reading of the paper. I was partially supported by NSF grant DMS-2154408.}

\bibliographystyle{amsplain}

\end{document}